\documentclass[twoside, 12pt]{article}
\usepackage{times}
\usepackage[english]{babel}
\usepackage{amssymb,amsmath}
\numberwithin{equation}{section}
\usepackage{amsthm}
\usepackage{array}
\usepackage{wasysym}
\usepackage{hyperref}
\usepackage[height=22cm , width = 16cm , top = 4cm , left = 3cm, a4paper]{geometry}
\usepackage{amssymb}
\usepackage{amsmath}
\usepackage{epsfig}
\usepackage{verbatim}
\usepackage{multicol}
\usepackage{graphicx, color, psfrag}
\usepackage[numbers,sort]{natbib}
\newtheorem{theorem}{Theorem}[section]
\newtheorem{remark}[theorem]{Remark}

\theoremstyle{definition}

\newcommand{\e}{\end{document}}

\begin{document}

\thispagestyle{empty}
\author{
{{\bf Beih S. El-Desouky$^1$, Abdelfattah Mustafa$^1$ and Nenad P. Caki\'{c}$^2$}
\newline{\it{{}  }}
 } { }\vspace{.2cm}\\
 \small \it $^1$Department of Mathematics, Faculty of Science, Mansoura University, Mansoura 35516, Egypt.\\
 \small \it $^2$Department of Mathematics, Faculty of Electrical Engineering, University of Belgrade, Serbia.
}

\title{New Results on Higher-Order Changhee Numbers and Polynomials}

\date{}

\maketitle
\small \pagestyle{myheadings}
        \markboth{{\scriptsize New Results on Higher-Order Changhee Numbers and Polynomials}}
        {{\scriptsize {Beih S. El-Desouky, Abdelfattah Mustafa and Nenad P. Caki\'{c}}}}

\hrule \vskip 8pt

\begin{abstract}
We derive new matrix representation for higher-order changhee numbers and polynomials. This helps us to obtain simple and short proofs of many previous results on higher-order changhee numbers and polynomials. Moreover, we obtain recurrence relations, explicit formulas and some new results for these numbers and polynomials. Furthermore, we investigate the relations between these numbers and polynomials and Stirling numbers, N\"{o}rlund and Bernoulli numbers of higher-order. Some numerical results using Mathcad program are introduced.
\end{abstract}




\section{Introduction}
For $\alpha \in \mathbb{N}$,  the Bernoulli polynomials of order $\alpha$ are defined by, see \cite{Carlitz1961}-\cite{Ozdenetal2009}
\begin{equation}\label{eq:1.1}
\left(\frac{t}{e^t-1}\right)^{\alpha} e^{xt}= \sum_{n=0}^\infty B_n^{(\alpha) } (x) \frac{  t^n}{n!}.
\end{equation}

\noindent
When $x=0,B_n^{(\alpha) }=B_n^{(\alpha)} (0)$ are the Bernoulli numbers of order $\alpha$, defined by
\begin{equation}\label{eq:1.2}
\left( \frac{t}{e^t-1} \right)^{\alpha}=\sum_{n=0}^\infty B_n^{(\alpha) }  \frac{ t^n}{n!}.
\end{equation}

\noindent
The Euler polynomial of order $\alpha (\in \mathbb{N})$ are defined by the generating function to be
\begin{equation} \label{eq:1.3}
\left(\frac{2}{e^t+1} \right)^{\alpha} e^{xt} =\sum_{n=0}^\infty E_n^{(\alpha)} (x) \frac{t^n}{n!},
\end{equation}
see \cite{BayadandKim2010,AraciandAcikgoz2012}, when $x = 0,E_n^{(\alpha)} = E_n^{(\alpha)}(0)$ are called the Euler numbers of order $\alpha$.

\noindent
Luo \cite[Eq. 2.8]{Luo2009}, introduced the following explicit formula for the Euler polynomial with order $\alpha$,
\begin{equation} \label{1.4}
E_n^{(\alpha)} (x) = \sum_{i=0}^n \binom{n}{i} x^{n-i} \sum_{j=0}^i \left(\frac{-1}{2}\right)^j \binom{\alpha+j-1}{j} s_2(i,j),
\end{equation}
where $s_2(i,j)$, are the Stirling numbers of the second kind.

\noindent
We can represent Euler polynomials of order $\alpha$,  by $(n+1)\times(\alpha+1)$ matrix , $0\leq \alpha \leq n$, as follows
{\scriptsize
\begin{equation} \label{EPH}
{\textbf {E} }^{(\alpha)} (x)=
\left(
\begin{array}{ccccc}
E_0^{(0)}(x) & E_0^{(1)} (x) & E_0^{(2)} (x) & \cdots & E_0^{(\alpha)} (x) \\
E_1^{(0)} (x) & E_1^{(1)} (x) & E_1^{(2)} (x) & \cdots & E_1^{(\alpha)} (x) \\
\vdots    & \vdots    & \vdots    & \ddots & \vdots    \\
E_n^{(0) } (x)& E_n^{(1) } (x) & E_n^{(2) } (x) & \cdots & E_n^{(\alpha) } (x)
\end{array}
\right).
\end{equation}
}

\noindent
For example, if setting  $0\leq n \leq 4, \; 0 \leq \alpha \leq n$ , in (\ref{EPH}), we have the following Euler polynomial, for $n, \alpha=0, 1,\cdots, 4.$
{\scriptsize
\[
{\textbf {E}}^{(\alpha)} (x)=
\left(
\begin{array}{ccccc}
0	& 1	& 1	& 1	& 1 \\
0	& x-\frac{1}{2}	& x-1	& x-\frac{3}{2}	& x-2\\
0	& x^2-x	& x^2-2x+\frac{1}{2}	& x^2-3x+\frac{3}{2}	& x^2-4x+3\\
0	& x^3-\frac{3x^2}{2}+\frac{1}{4}	& x^3-3x^2+\frac{3x}{2}+\frac{1}{2}	& x^3-\frac{9x^2}{2}+\frac{9x}{2}	& x^3-6x^2+9x-2\\
0	& x^4-2x^3+x	& x^4-4x^3+3x^2+2x-1	& x^4-6x^3+9x^2-3	& x^4-8x^3+18x^2-8x-\frac{9}{2}\\
\end{array}
\right).
\]
}

\noindent
Kim et al. \cite{Kimetal2014C} define the Changhee polynomials of the first kind with order $k$ by the generating function as follows.
\begin{equation} \label{eq:1.5}
\left(\frac{2}{2+t} \right)^k (1+t)^x = \sum_{n=0}^\infty Ch_n^{(k)} (x) \frac{t^n}{n!}.
\end{equation}

\noindent
If $x=0$, hence  $Ch_n^{(k)}=Ch_n^{(k)} (0)$ are called the Changhee numbers of the first kind with order $k$.
For $k=1$, $Ch_n (x)=Ch_n^{(1)} (x)$ are called the Changhee polynomial of the first kind with order $1$, that is defined as, see \cite{Kimetal2013, Kim2004}.
\begin{equation} \label{eq:1.6}
\left( \frac{2}{t+2} \right) (1+t)^x = \sum_{n=0}^\infty Ch_n (x) \frac{t^n}{n!}.
\end{equation}

\noindent
Moreover, Kim et al. \cite{Kimetal2014D} defined the Daehee polynomials of order $k$ by the generating function as follows
\begin{equation} \label{eq:1.7}
\sum_{n=0}^\infty D_n^{(k) }  (x)  \frac{t^n}{n!}=\left(\frac{\log(1+t)}{t}\right)^k (1+t)^x.
\end{equation}

\noindent
In the special case, $x=0,D_n^{(k)}=D_n^{(k)} (0)$ are called the Daehee numbers of order $k$, defined by
\begin{equation} \label{eq:1.8}
\left( \frac{\log(1+t)}{t} \right)^k=\sum_{n=0}^\infty D_n^{(k)}  \frac{ t^n}{n!}.
\end{equation}

\noindent
When $k=1,\; D_n(x)=D_n^{(1)} (x)$ are the Daehee polynomials of the first kind with order 1, defined by
\begin{equation} \label{eq:1.9}
\left( \frac{\log(1+t)}{t} \right) (1+t)^x= \sum_{n=0}^\infty D_n (x) \frac{t^n}{n!},
\end{equation}
see \cite{KimKim2013, KimSimsek2008} and  \cite{Ozdenetal2009}.

\noindent
The Stirling numbers of the first and second kind are defined, respectively, by
\begin{equation} \label{eq:1.10}
(x)_n = \prod_{i=0}^{n-1} (x-i) = \sum_{l=0}^n s_1 (n,l) x^l,
\end{equation}

\noindent
where $s_1(n,0)=\delta_{n,0}, \; s_1(n,k)=0, \; \mbox{for} \; k>n,$ and
\begin{equation} \label{eq:1.11}
x^n =  \sum_{k=0}^n s_2 (n,k) (x)_k,
\end{equation}

\noindent
where $s_2(n,0)=\delta_{n,0}, \; s_2(n,k)=0, \; \mbox{for} \; k>n$, and $\delta_{n,k}$ is the kronecker delta.

\noindent
The Stirling numbers of the second kind  have the generating function, see  \cite{Carlitz1961,Comtet1974,El-Desouky1994,El-Desoukyetal2010} and \cite{Gould1972}
\begin{equation} \label{eq:1.12}
\left(e^t-1 \right)^m = m! \sum_{l=m}^\infty s_2 (l,m) \frac{  t^l}{l!}.
\end{equation}

\section{Higher-order Changhee Polynomials of The First Kind}
In this section, we derive an explicit formulas and recurrence relations for the higher-order Changee numbers and polynomials of the first kind. Also the relation between these numbers and N\"{o}rlund numbers are given. Furthermore, we introduce the matrix representation of some results for higher-order Changhee numbers and polynomial obtained by Kim et al. \cite{Kimetal2014C} in terms of Stirling numbers, N\"{o}rlund numbers, Euler and Bernoulli numbers of higher-order and investigate a simple and short proofs of these results.

\noindent
Kim et al. \cite[ Eq. 2.8]{Kimetal2014C} proved that, for $n\in \mathbb{Z}, \; k\in \mathbb{N}$, the Chaghee higher order numbers of the first kind can be obtained as
\begin{eqnarray} \label{eq:2.1}
2^n Ch_n^{(k) }&=& (-1)^n (k+n-1)_n
\nonumber\\
& = & (-1)^n \sum_{\ell=0}^{n} s_1(n,l) (k+n-1)^{\ell}.
\end{eqnarray}

\noindent
If $k=1$, $Ch_n=(-1)^n {n! \over 2^n}.$

\noindent
We can represent the Changee numbers of the first kind of order $k$,  by $(n+1)\times(k+1)$ matrix , $0\leq k \leq n$, as follows
{\scriptsize
\[
{\textbf {Ch} }^{(k) }=
\left(
\begin{array}{ccccc}
Ch_0^{(0)} & Ch_0^{(1)} & Ch_0^{(2)} & \cdots & Ch_0^{(k)} \\
Ch_1^{(0)} & Ch_1^{(1)} & Ch_1^{(2)} & \cdots & Ch_1^{(k)} \\
\vdots    & \vdots    & \vdots    & \ddots & \vdots    \\
Ch_n^{(0) }& Ch_n^{(1) }& Ch_n^{(2) } & \cdots & Ch_n^{(k) }
\end{array}
\right).
\]
}

\noindent
For example, if setting  $0\leq n \leq 4, \; 0 \leq k \leq n$ , in (\ref{eq:2.1}), we have the following higher order Changhee numbers of the first kind, for $n,k=0, 1,\cdots, 4.$
{\scriptsize
\[
{\textbf {Ch}}^{(k)}=
\left(
\begin{array}{ccccc}

0	& 1	    &  1	& 1	& 1 \\
0	& -1/2	& -1	& -3/2	&-2 \\
0	& 1/2	& 3/2	& 3	& 5 \\
0	& -3/4	& -3	& -15/2	& -15 \\
0	& 3/2	& 15/2	& 45/2	& 105/2
\end{array}
\right).
\]
}

\noindent
Kim et al. \cite[Theorem 1]{Kimetal2014D}, derived the following relation
\begin{equation} \label{eq:2.2}
D_n^{(k)}= \frac{s_1(n+k,k)}{\binom{n+k}{k}}.
\end{equation}

\noindent
Substituting from Eq. (\ref{eq:2.2}) into Eq. (\ref{eq:2.1}), we obtain the relation between  higher order Changhee numbers and higher order Dahee numbers of the first kind as follows
\begin{equation} \label{eq:2.3}
2^n Ch_n^{(k)}=(-1)^n \sum_{\ell=0}^{n} \binom{n}{\ell} n^{\ell} D_{n-\ell}^{(\ell)}.
\end{equation}

\noindent
Setting  $k=1$ into Eq. (\ref{eq:2.1}), we get the explicit form for the Changhee number of the first kind as follows. 

\begin{equation}\label{eq:2.4}
Ch_n=\frac{(-1)^n n!}{2^n}.
\end{equation}

\noindent
El-Desouky and Mustafa \cite[Eq. 6]{El-DesoukyMustafa2014}, derived the explicit form for the Daehee numbers of the first kind as follows
$$
D_n=(-1)^n \frac{n!}{n+1},
$$

\noindent
then the relation between Daehee numbers and Changhee numbers can be obtained as

\begin{equation} \label{eq:2.5}
D_n=\frac{2^n}{n+1} Ch_n.
\end{equation}

\noindent
By using the recurrence relation for the Daehee numbers of the first kind, \cite{El-DesoukyMustafa2014},  we can derive the recurrence relation for the Changhee numbers of the first kind as follows
\begin{equation} \label{eq:2.6}
2 Ch_n+n \, Ch_{n-1}=0.
\end{equation}

\noindent
The N\"{o}rlund numbers of the second kind have the explicit formula \cite[Remark 4]{LiuSrivastava2006},
\begin{equation} \label{eq:2.7}
b_n^{(-1)}={(-1)^n \over {n+1}}.
\end{equation}

\noindent
From Eq. (\ref{eq:2.4}) and Eq.(\ref{eq:2.7}), we can obtain the relation between the Changhee numbers and the N\"{o}rlund numbers as follows
\begin{equation} \label{eq:2.8}
Ch_n={(n+1)! \over 2^n} \; b_n^{(-1)}.
\end{equation}

\noindent
Kim et al. \cite[Theorem 2.2]{Kimetal2014C},  proved the following result. For $n \geq 0 $,
\begin{equation}\label{eq:2.9}
Ch_n^{(k)}=\sum_{\ell=0}^n s_1 (n, \ell) E_{\ell}^{(k)}.
\end{equation}

\noindent
\begin{remark}
We can write this relation in the matrix form as follows.
\begin{equation}\label{eq:2.10}
{\bf Ch}^{(k)}= {\bf S}_1 \, {\bf E}^{(k)},
\end{equation}
\end{remark}

\noindent
where ${\bf Ch}^{(k) }$  is $(n+1)\times (k+1), \, 0\leq k\leq n,$ matrix for the Changhee numbers of the first kind of order $k$ and  ${\bf S}_1$ is  $(n+1)\times(k+1)$ lower triangular matrix for the Strirling numbers of the first  kind and ${\bf E}^{(k) }$ is $(n+1)\times(k+1), \, 0\leq k \leq n,$ matrix for the Euler numbers of order $k$.\\
For example, if setting  $0\leq n \leq 4, \; 0 \leq k \leq n$ , in (\ref{eq:2.10}), we have
{\scriptsize
\[
\left(
\begin{array}{ccccc}
1	& 0	 & 0	& 0	 & 0\\
0	& 1	 & 0	& 0	 & 0\\
0	& -1 & 1	& 0	 & 0\\
0	& 2	 & -3	& 1	 & 0\\
0	& -6 & 11	& -6 & 1\\
\end{array}
\right)\left(
\begin{array}{ccccc}
0	& 1	    &  1	& 1	   & 1\\
0	& -1/2	& -1	& -3/2 & -2\\
0	& 0	    & 1/2	& 3/2  & 3\\
0	& 1/4	& 1/2	& 0	  & -2\\
0	& 0	    & -1	& -3  &-9/2\\
\end{array}
\right)=\left(
\begin{array}{ccccc}
0	& 1	    & 1	    & 1	    & 1\\
0	& -1/2	& -1	& -3/2	& -2\\
0	& 1/2	& 3/2	& 3	    & 5\\
0	& -3/4	& -3	& -15/2	& -15\\
0	& 3/2	& 15/2	& 45/2	& 105/2\\
\end{array}
\right).
\]
}

\noindent
Kim et al. \cite[Theorem 2.3]{Kimetal2014C} proved the following result. For $n \geq 0$, we have
\begin{equation} \label{eq:2.11}
E_m^{(k)}= \sum_{n=0}^m  s_2 (n,m) Ch_n^{(k)}.
\end{equation}

\noindent
We can write this relation in the matrix form as follows
\begin{equation} \label{eq:2.12}
{\bf E}^{(k)}={\bf S}_2 \, {\bf Ch}^{(k)},
\end{equation}

\noindent
where  ${\bf S}_2$ is $(n+1)\times(k+1)$ lower triangular matrix for the Strirling numbers of the second kind.

\noindent
For example, if setting  $0\leq n \leq 4, \; 0 \leq k \leq n$ , in (\ref{eq:2.12}), we have

{\scriptsize
\[
\left(
\begin{array}{rrrrr}
1	& 0	& 0	& 0	& 0\\
0	& 1	& 0	& 0	& 0\\
0	& 1	& 1	& 0	& 0\\
0	& 1	& 3	& 1	& 0\\
0	& 1	& 7	& 6	& 1\\
\end{array}
\right)\left(
\begin{array}{rrrrr}
0	& 1	& 1	& 1	& 1\\
0	& -\frac{1}{2}	& -1	& -\frac{3}{2}	& -2\\
0	& \frac{1}{2}	& \frac{3}{2}	& 3	& 5\\
0	& -\frac{3}{4}	& -3	& -\frac{15}{2}	& -15\\
0	& \frac{3}{2}	& \frac{15}{2}	& \frac{45}{2}	& \frac{105}{2}\\
\end{array}
\right)=
\left(
\begin{array}{rrrrr}
0	& 1	& 1	& 1	& 1\\
0	& -\frac{1}{2}	& -1	& -\frac{3}{2}	& -2\\
0	& 0	& \frac{1}{2}	& \frac{3}{2}	& 3\\
0	& \frac{1}{4}	& \frac{1}{2}	& 0	& -2\\
0	& 0	& -1	& -3	& -\frac{9}{2}\\
\end{array}
\right).
\]
}

\begin{remark}
Using the matrix form (\ref{eq:2.10}), we easily can derive a short proof of  \cite[Theorem 2.3]{Kimetal2014C}. Multiplying both sides by the Striling number of second kind as follows.
\[
{\bf S}_2 \, {\bf Ch}^{(k) }= {\bf S}_2 \, {\bf S}_1 \,  {\bf E}^{(k)}={\bf I\, E}^{(k)}={\bf E}^{(k) },
\]

\noindent
where {\bf I} is the identity matrix of order $(n+1)$.
\end{remark}

\noindent
Kim et al. \cite{Kimetal2013} introduced the following definition for the the Changhee Polynomial of the first kind.
\begin{equation} \label{eq:2.13}
\sum_{n=0}^\infty
Ch_n(x) {t^n \over n!} =\left({2 \over {2+t}} \right) (1+t)^x.
\end{equation}

\begin{theorem}
For $n  \in \mathbb{N}$, the Changhee polynomials satisfy the following relation
\begin{equation} \label{eq:2.14}
Ch_n(x) =\sum_{i=0}^n (-1)^i  {n! \over 2^i}  \binom{x}{n-i}.
\end{equation}	
\end{theorem}

\begin{proof}
From Eq. (\ref{eq:2.13}), we have	
\begin{eqnarray} \label{eq:2.15}
\sum_{n=0}^\infty
Ch_n(x) {t^n \over n!} &=&\left(1+{t \over 2} \right)^{-1} (1+t)^x
\nonumber\\
&= &
\sum_{i=0}^\infty (-1)^i \left( {t \over 2} \right)^i \sum_{j=0}^\infty \binom{x}{j} t^j 
\nonumber\\
&= &
\sum_{i=0}^\infty \sum_{j=0}^\infty (-1)^i  {1 \over 2^i}  \binom{x}{j}   t^{j+i} 
\nonumber\\
&= &
\sum_{i=0}^\infty \sum_{n=i}^\infty (-1)^i  {n! \over 2^i}  \binom{x}{n-i}   {t^n \over n!} 
\nonumber\\
&= &
\sum_{n=0}^\infty \sum_{i=0}^n (-1)^i  {n! \over 2^i}  \binom{x}{n-i}   {t^n \over n!} ,
\end{eqnarray}

\noindent
by equating the coefficient of $t^n$, in both sides of Eq.(\ref{eq:2.15}), we have
\begin{equation*}
Ch_n(x) =\sum_{i=0}^n (-1)^i  {n! \over 2^i}  \binom{x}{n-i}.
\end{equation*}

\noindent
This completes  the proof.
	
\end{proof}

\noindent
Kim et al. \cite{Kimetal2014C} introduced the following definition for the Changhee Polynomials of the first kind with order $k$ by
\begin{equation} \label{eq:2.16}
\sum_{n=0}^\infty
Ch_n^{(k)} (x) {t^n \over n!} =\left({2 \over {2+t}} \right)^k (1+t)^x.
\end{equation}

\begin{theorem}
For $n,k  \in \mathbb{N}$, the Changee polynomials of higher order satisfy the following relation
\begin{equation} \label{eq:2.17}
Ch_n^{(k)} (x) =\sum_{i=0}^n (-1)^i  {n! \over 2^i}  \binom{k+i-1}{i}\binom{x}{n-i}.
\end{equation}
\end{theorem}

\begin{proof}
From Eq. (\ref{eq:2.16}), we have

\begin{eqnarray} \label{eq:2.18}
\sum_{n=0}^\infty
Ch_n^{(k)}(x) {t^n \over n!} &=&\left(1+{t \over 2} \right)^{-k} (1+t)^x
\nonumber\\
&= &
\sum_{i=0}^\infty (-1)^i \binom{k+i-1}{i} \left( {t \over 2} \right)^i \sum_{j=0}^\infty \binom{x}{j} t^j 
\nonumber\\
&= &
\sum_{i=0}^\infty \sum_{j=0}^\infty (-1)^i  {1 \over 2^i}  \binom{k+i-1}{i}\binom{x}{j}   t^{j+i} 
\nonumber\\
&= &
\sum_{i=0}^\infty \sum_{n=i}^\infty (-1)^i  {n! \over 2^i}  \binom{k+i-1}{i}\binom{x}{n-i}   {t^n \over n!} 
\nonumber\\
&= &
\sum_{n=0}^\infty \sum_{i=0}^n (-1)^i  {n! \over 2^i}  \binom{k+i-1}{i}\binom{x}{n-i}   {t^n \over n!} ,
\end{eqnarray}

\noindent
by equating the coefficient of $t^n$, on both sides of Eq.(\ref{eq:2.18}), we have
\begin{equation} \label{eq:2.19}
Ch_n^{(k)} (x) =\sum_{i=0}^n (-1)^i  {n! \over 2^i}  \binom{k+i-1}{i}\binom{x}{n-i}.
\end{equation}

\noindent
This completes the proof.

\end{proof}

\begin{remark}
Setting $k=1$ in Eq. (\ref{eq:2.18}), we get  Eq. (\ref{eq:2.15}).
\end{remark}

\noindent
Kim et al. \cite[Corollary 2.4]{Kimetal2014C} introduced the following result for the Changhee Polynomials of the first kind with oreder $k$.

\noindent 
For $n \geq 0$, we have
\begin{equation} \label{eq:2.20}
Ch_n^{(k)} (x) = \sum_{\ell=0}^n  s_1 (n,\ell) E_{\ell}^{(k)} (x).
\end{equation}

\noindent
We can write this relation in the matrix form as follows
\begin{equation} \label{eq:2.21}
{\bf Ch}^{(k)} (x) = {\bf S}_1 \,  {\bf E}^{(k)} (x),
\end{equation}

\noindent
where ${\bf Ch}^{(k) } (x)$ is $(n+1)\times (k+1)$ matrix for the Changhee polynomials of the first kind with order $k$ and  ${\bf E}^{(k) } (x)$ is $(n+1)\times(k+1)$ matrix for the Euler polynomials of order $k$.

\noindent
For example, if setting  $0 \leq n \leq 3, \; 0 \leq k \leq n$ , in (\ref{eq:2.21}), we have
{\scriptsize
\begin{eqnarray*}
\left(
\begin{array}{cccc}
1	& 0	& 0	& 0	\\
0	& 1	& 0	& 0	\\
0	& -1	& 1	& 0	\\
0	& 2	& -3	& 1	 \\
\end{array}
\right)
\left(
\begin{array}{cccc}
0	& 1	& 1	& 1	\\
0	& x-\frac{1}{2}	& x-1	& x-\frac{3}{2}	\\
0	& x^2-x	& x^2-2x+\frac{1}{2}	& x^2-3x+\frac{3}{2}\\
0	& x^3-\frac{3}{2}x^2+\frac{1}{4}	& x^3-3*x^2+\frac{3}{2}x+\frac{1}{2}	& x^3-\frac{9}{2}x^2+\frac{9}{2}x	\\
\end{array}
\right)=
\nonumber\\
\left(
\begin{array}{cccc}
0	& 1	& 1	& 1	\\
0	& x-\frac{1}{2}	& x-1	& x-\frac{3}{2}	\\
0	& x^2-2x+\frac{1}{2}	& x^2-3x+\frac{3}{2}	& x^2-4x+3	\\
0	& x^3-\frac{9}{2}x^2+5x-\frac{3}{4}	& x^3-6x^2+\frac{19}{2} x-3	& x^3-\frac{15}{2}x^2 + \frac{31}{2}x-\frac{15}{2}	\\
\end{array}
\right).
\end{eqnarray*}
}

\noindent
Kim et al. \cite[Theorem 2.5]{Kimetal2014C} introduced the following result. 
For $m \geq 0$, we have
\begin{equation} \label{eq:2.22}
E_m^{(k) } (x) = \sum_{n=0}^m Ch_n^{(k) } (x) s_2 (m,n).
\end{equation}

\noindent
We can write Eq. (\ref{eq:2.22}) in the matrix form as follows
\begin{equation}\label{eq:2.23}
{\bf E}^{(k) } (x)={\bf S}_2 \, {\bf Ch}^{(k) } (x).
\end{equation}

\noindent
For example, if setting  $0 \leq n \leq 3, \; 0 \leq k \leq n$ , in (\ref{eq:2.23}), we have
{\scriptsize
\begin{eqnarray*}
\left(
\begin{array}{cccc}
1 & 0 & 0 & 0 \\
0 & 1 & 0 & 0 \\
0 & 1 & 1 & 0 \\
0 & 1 & 3 & 1
\end{array}
\right)
\left(
\begin{array}{cccc}
0	& 1	& 1	& 1 \\
0	& x-\frac{1}{2}	& x-1	& x-\frac{3}{2} \\
0	& x^2-2x+\frac{1}{2}	& x^2-3x+\frac{3}{2}	& x^2-4x+3\\
0	& x^3-\frac{9}{2}x^2+5x-\frac{3}{4}	& x^3-6x^2+\frac{19}{2}x-3	& x^3-\frac{15}{2}x^2+\frac{31}{2}x-\frac{15}{2}
\end{array}
\right)=
\\
\left(
\begin{array}{cccc}
0	& 1	 & 1	&  1  \\
0	& x-\frac{1}{2}	& x-1	& x-\frac{3}{2}	\\
0	& x^2-x	& x^2-2x+\frac{1}{2}	& x^2-3x+\frac{3}{2}	\\
0	& x^3-\frac{3}{2}x^2+\frac{1}{4}	& x^3-3x^2+\frac{3}{2}x+\frac{1}{2}	& x^3-\frac{9}{2}x^2+\frac{9}{2}x
\end{array}
\right).
\end{eqnarray*}
}

\begin{remark} We can prove \cite[Theorem 2.5]{Kimetal2014C} by using the matrix form (\ref{eq:2.21}) as follows. Multiplying both sides of (\ref{eq:2.21}) by the Striling number of second kind, we have
\[
{\bf S}_2 \, {\bf Ch}^{(k)} (x) ={\bf S}_2 \, {\bf S}_1 \,  {\bf E}^{(k)} (x)={\bf I\, E}^{(k)} (x)={\bf E}^{(k)} (x).
\]
\end{remark}

\noindent
We can determine the relation between the Daehee numbers of the first kind, Euler's polynomials  and the Changhee polynomials of the first kind with higher order as follows.

\noindent
From Eq. (\ref{eq:2.2}), we get
\begin{equation} \label{eq:2.24}
s_1(n,\ell)=\binom{n}{\ell} D_{n-k}^{(\ell)},
\end{equation}

\noindent
substituting from Eq. (\ref{eq:2.24}) into Eq. (\ref{eq:2.9}), we obtain
\begin{equation} \label{eq:2.25}
Ch_n^{(k)}(x)= \sum_{\ell=0}^n \binom{n}{\ell} D_{n-\ell}^{(\ell)} E_{\ell}^{(k)}(x).
\end{equation}

\noindent
Also, from Kim et al. (2014), Theorem 3, we have
\begin{equation} \label{eq:2.26}
D_{n-k}^{(k)}=\sum_{m=0}^{n-k} s_1(n-k,m) B_m^{(k)}.
\end{equation}

\noindent
Substituting from Eq. (\ref{eq:2.26}) into Eq. (\ref{eq:2.25}), we get
\begin{equation} \label{eq:2.27}
Ch_n^{(k)}(x)= \sum_{\ell=0}^n  \sum_{m=0}^{n-\ell} \binom{n}{\ell} s_1(n-\ell,m) B_m^{(\ell)} E_{\ell}^{(k)}(x).
\end{equation}

\section{Higher-order Changhee Polynomials of The Second Kind}
In this section, we derive an explicit formulas and recurrence relations for the higher-order Changee numbers and polynomials of the second kind. Also the relation between these numbers and N\"{o}rlund numbers are given. Furthermore, we introduce the matrix representation of some results for higher-order Changhee numbers and polynomial obtained by Kim et al. \cite{Kimetal2014C} in terms of Stirling numbers, N\"{o}rlund numbers, Euler and Bernoulli numbers of higher-order and investigate a simple and short proofs of these results.\\

\noindent
Kim et al. \cite[Eq. 2.8]{Kimetal2014C}, defined the Changhee numbers of the second Kind with the second kind with order $k(\in  \mathbb{N})$ as follows:
\begin{equation} \label{eq:3.1}
\sum_{n=0}^\infty \widehat{Ch}_n^{(k)} \frac{t^n}{n!} =\left(\frac{2}{2+t} \right)^k (1+t)^k.
\end{equation}

\noindent
Kim et al. \cite[Eq. 2.24]{Kimetal2014C} proved that, for $n\geq 0$, the Chaghee numbers of the second kind satisfy the following relation:
\begin{equation} \label{eq:3.2}
\widehat{Ch}_n^{(k) }=\sum_{\ell}^n (-1)^{\ell} s_1(n, \ell) E_{\ell}^{(k)} .
\end{equation}

\noindent
We can write Eq. (\ref{eq:3.2}) in the matrix form as follows
\begin{equation} \label{eq:3.3}
\widehat{{\textbf {Ch}} }^{(k) }= {\bf S}_1 {\bf D} {\bf E}^{(k)},
\end{equation}

\noindent
where $\widehat{{\bf Ch}}^{(k) }$  is $(n+1)\times(k+1)$  matrix of Changhee numbers of the second kind with order $k$ and ${\bf \bf{D}}$ is $(n+1)\times(n+1)$ diagonal matrix with element $D_{ii}=(-1)^i, \; i=0, 1, \cdots,n$,
{\scriptsize
\[
\widehat{{\textbf {Ch}} }^{(k) }=
\left(
\begin{array}{ccccc}
\widehat{Ch}_0^{(0)} & \widehat{Ch}_0^{(1)} & \widehat{Ch}_0^{(2)} & \cdots & \widehat{Ch}_0^{(k)} \\
\widehat{Ch}_1^{(0)} & \widehat{Ch}_1^{(1)} & \widehat{Ch}_1^{(2)} & \cdots & \widehat{Ch}_1^{(k)} \\
\vdots    & \vdots    & \vdots    & \ddots & \vdots    \\
\widehat{Ch}_n^{(0) }& \widehat{Ch}_n^{(1) }& \widehat{Ch}_n^{(2) } & \cdots & \widehat{Ch}_n^{(k) }
\end{array}
\right).
\]
}

\noindent
For example, if setting  $0\leq n \leq 3, \; 0 \leq k \leq n$ , in (\ref{eq:3.3}), we have
{\scriptsize
\[
\left(
\begin{array}{rrrr}
0	& 1	 & 1	& 1\\
0	& \frac{1}{2} &	1	& \frac{3}{2}\\
0	& -\frac{1}{2} & 	-\frac{1}{2} &	0\\
0	& \frac{3}{4} &	0	& -\frac{3}{2}
\end{array}
\right)=
\left(
\begin{array}{rrrr}
1	& 0	 & 0	& 0\\
0	& 1	 & 0	& 0\\
0	& -1 & 1	& 0\\
0	& 2	& -3	& 1
\end{array}
\right)
\left(
\begin{array}{rrrr}
1	& 0	 & 0	& 0\\
0	& -1 & 0	& 0\\
0	& 0	 & 1	& 0\\
0	& 0	 & 0	& -1
\end{array}
\right)
\left(
\begin{array}{rrrr}
0	& 1	& 1	& 1\\
0	& -\frac{1}{2}	& -1	& -\frac{3}{2} \\
0	& 0	&  \frac{1}{2}	& \frac{3}{2} \\
0	& \frac{1}{4}	& \frac{1}{2}	& 0
\end{array}
\right).
\]
}

\noindent
Kim et al. \cite[ Eq. 2.25]{Kimetal2014C} proved that, for $n\geq 0$, the Chaghee numbers of the second kind satisfies the following relation
\begin{equation} \label{eq:3.4}
\sum_{n=0}^m
\widehat{Ch}_n^{(k)} {t^n \over n!}=\left({2\over {2+t}} \right)^k (1+t)^k.
\end{equation}

\begin{theorem}
For $n, k \in \mathbb{N}$, the Changhee numbers of the second kind satisfy the following relation
\begin{equation} \label{eq:3.5}
\widehat{Ch}_n^{(k)}= n! \sum_{i=0}^n { (-1)^i \over {2^i}}  \binom{k+i-1}{i}  \binom{k}{n-i}. 
\end{equation}
\end{theorem}

\begin{proof}
From Eq.(\ref{eq:3.4}), then
\begin{eqnarray*}
	\sum_{n=0}^\infty  \widehat{Ch}_n^{(k)} {t^n \over n!} & = &
	\left(1+ {t\over 2} \right)^{-k} (1+t)^k
	\nonumber\\
	& = &
	\sum_{i=0}^\infty (-1)^i \binom{k+i-1}{i} \left( {t \over 2}\right)^i \sum_{j=0}^\infty \binom{k}{j} t^j 
	\nonumber\\
	&= &
	\sum_{i=0}^\infty \sum_{j=0}^\infty 
	{ (-1)^i \over {2^i}} \binom{k+i-1}{i}  \binom{k}{j}  t ^{i+j} 
	\nonumber\\
	& = &
	\sum_{i=0}^\infty \sum_{\ell=i}^\infty 
	{ (-1)^i \over {2^i}} \binom{k+i-1}{i}  \binom{k}{\ell-i}  t^\ell 
	\nonumber\\
	& = &
	\sum_{\ell=0}^\infty \sum_{i=0}^\ell 
	{ (-1)^i \over {2^i}} \ell! \binom{k+i-1}{i}  \binom{k}{\ell-i}  {t^\ell \over \ell!}, 
\end{eqnarray*}

\noindent
by equating the coefficients of $t^\ell$ on both sides, we have Eq. (\ref{eq:3.5}). This completes the proof.
\end{proof}

\noindent
Kim et al.  \cite[Theorem 2.7]{Kimetal2014C} introduced the following result. For  $n\geq 0$,
\begin{equation} \label{eq:3.6}
E_m^{(k)}(k)=\sum_{n=0}^m \widehat{Ch}_n^{(k)} S_2(m,n),
\end{equation}

\noindent
we can write (\ref{eq:3.6}) in the matrix form as follows
\begin{equation}\label{eq:3.7}
{\bf E}^{(k)}(k)= {\bf S}_2 \, \widehat{\bf Ch}^{(k) },
\end{equation}

\noindent
where ${\bf E}^{(k) } (k)$  is $(n+1)\times(k+1)$  matrix of Euler polynomial when $x=k$.

\noindent
For example, if setting $0 \leq n \leq 3,\; 0 \leq k \leq n$ in (\ref{eq:3.7}), we have
{\scriptsize
\[
\left(
\begin{array}{cccc}
0	& 1	& 1	& 1\\
0	& \frac{1}{2}	& 1	& \frac{3}{2} \\
0	& 0	 & \frac{1}{2} &	\frac{3}{2} \\
0	& -\frac{1}{4}  & 	-\frac{1}{2} & 	0
\end{array}
\right)=
\left(
\begin{array}{cccc}
1 & 0 & 0 & 0 \\
0 & 1 & 0 & 0 \\
0 & 1 & 1 & 0 \\
0 & 2 & 3 & 1
\end{array}
\right)\left(
\begin{array}{cccc}
0	& 1	& 1	& 1\\
0	& 1/2 &	1	& 3/2\\
0	& -1/2 &	-1/2 &	0\\
0	& 3/4	& 0	& -3/2
\end{array}
\right).
\]
}

\noindent
Kim et al.  \cite{Kimetal2014C} defined the Changhee polynomials of the second kind with order $k(\in \mathbb{N})$ as follows:
\begin{equation} \label{eq:3.8}
\sum_{n=0}^\infty \widehat{Ch}_n^{(k)} (x) = \left( \frac{2}{2+t} \right)^k (1+t)^{x+k}.
\end{equation}

\noindent
Kim et al.  \cite[Eq. (2.34)]{Kimetal2014C} introduced the following result.
\begin{equation} \label{eq:3.9}
\widehat{Ch}_n^{(k) } (x)=\sum_{\ell}^n (-1)^{\ell} s_1(n, \ell) E_{\ell}^{(k)} (-x) .
\end{equation}

\noindent
We can write Eq. (\ref{eq:3.9}) in the matrix form as follows
\begin{equation} \label{eq:3.10}
\widehat{{\textbf {Ch}} }^{(k) } (x)= {\bf S}_1 {\bf D} {\bf E}^{(k)} (-x),
\end{equation}

\noindent
where $\widehat{{\bf Ch}}^{(k) } (x)$  is $(n+1)\times(k+1)$  matrix of Changhee polynomials of the second kind with order $k$.

\noindent
For example, if setting  $0\leq n \leq 3, \; 0 \leq k \leq n$ , in (\ref{eq:3.10}), we have 
{\scriptsize
\begin{eqnarray*}
	\left(
	\begin{array}{cccc}
		0	& 1	& 1	& 1	\\
		0	& x+\frac{1}{2}	& x+1	& x+\frac{3}{2}	\\
		0	& x^2-\frac{1}{2}	& x^2+x-\frac{1}{2}	& x^2+2x\\
		0	& x^3-\frac{3}{2}x^2-x+\frac{3}{4}	& x^3-\frac{5}{2}x	& x^3+\frac{3}{2}x^2-\frac{5}{2}x-\frac{3}{2}\\
	\end{array}
	\right)=
\left(
\begin{array}{cccc}
1	& 0	 & 0	& 0\\
0	& 1	 & 0	& 0\\
0	& -1 & 1	& 0\\
0	& 2	& -3	& 1
\end{array}
\right)
\left(
\begin{array}{cccc}
1	& 0	 & 0	& 0\\
0	& -1 & 0	& 0\\
0	& 0	 & 1	& 0\\
0	& 0	 & 0	& -1
\end{array}
\right)
\times
\\
\left(
\begin{array}{cccc}
0	& 1	& 1	& 1\\
0	& -x-\frac{1}{2} & 	-x-1	& -x-\frac{3}{2} \\
0	& x^2+x	& x^2+2x+\frac{1}{2} &	x^2+3x+\frac{3}{2} \\
0	& -\frac{3}{2}x^2-x^3+\frac{1}{4} &	-3x^2-\frac{3}{2}x-x^3+\frac{1}{2} &	-x^3-\frac{9}{2}x^2-\frac{9}{2}x
\end{array}
\right).
\end{eqnarray*}
}

\noindent
Kim et al.  \cite[Theorem 2.9]{Kimetal2014C} introduced the following result. For  $n\geq 0$,
\begin{equation} \label{eq:3.11}
E_m^{(k)}(x+k)=\sum_{n=0}^m \widehat{Ch}_n^{(k)} (x) S_2(m,n).
\end{equation}

\noindent
We can write (\ref{eq:3.11}) in the matrix form as follows
\begin{equation}\label{eq:3.12}
{\bf E}^{(k)}(x+k)= {\bf S}_2 \, \widehat{\bf Ch}^{(k) }(x).
\end{equation}

\noindent
For example, if setting $0 \leq n \leq 3,\; 0 \leq k \leq n$ in (\ref{eq:3.12}), we have
{\scriptsize
\begin{eqnarray*}
	\left(
	\begin{array}{cccc}
		0	& 1	& 1	& 1\\
		0	& x+\frac{1}{2} &	x+1	 & x+\frac{3}{2} \\
		0	& x^2+x	& x^2+2x+\frac{1}{2} &	x^2+3x+\frac{3}{2} \\
		0	& x^3+\frac{3}{2}x^2-\frac{1}{4} &	x^3+\frac{3}{2}x^2+\frac{3}{2}x-\frac{1}{2} &	x^3+\frac{9}{2}x^2+\frac{9}{2}x
	\end{array}
	\right)
	=
\left(
\begin{array}{cccc}
1 & 0 & 0 & 0 \\
0 & 1 & 0 & 0 \\
0 & 1 & 1 & 0 \\
0 & 2 & 3 & 1
\end{array}
\right)
\times\\
\left(
\begin{array}{cccc}
0	& 1	& 1	& 1\\
0	& x+\frac{1}{2}	& x+1	& x+\frac{3}{2} \\
0	& x^2-\frac{1}{2} &	x^2+x-\frac{1}{2} &	x^2+2x\\
0	& x^3-\frac{3}{2}x^2-x+\frac{3}{4} &	x^3-\frac{5}{2}x	& x^3+\frac{3}{2}x^2-\frac{5}{2}x-\frac{3}{2}
\end{array}
\right).
\end{eqnarray*}
}

\noindent
Kim et al.  \cite[Theorem 2.10]{Kimetal2014C} introduced the following result.
For  $n\in \mathbb{Z},\; k\in \mathbb{N}$,
\begin{equation} \label{eq:3.13}
\frac{(-1)^n}{n!} \widehat{Ch}_n^{(k)} (x) = \sum_{m=1}^n \frac{\binom{n-1}{n-m}}{m!}  Ch_m^{(k)}(-x).
\end{equation}

\noindent
We can write Eq. (\ref{eq:3.13}) in the matrix form as follows
\begin{equation} \label{eq:3.14}
{\bf D_n} \widehat{{\bf Ch}}^{(k)} (x)=   {\bf C} \; {{\bf Ch}}^{(k)}(-x),
\end{equation}

\noindent
where ${(\bf D_n)}^{-1}$ is $(n+1)\times(n+1)$ diagonal matrix with elements ${\bf D_n}_{ii}=(-1)^i/i!$, 
${\bf C}$ is $(n+1)\times(n+1)$ matrix with elements ${\bf C}_{ij}=\binom{i-1}{i-j}/j!$.

\noindent
For example, if setting  $0\leq n \leq 3,\; 0 \leq k \leq n$ in (\ref{eq:3.14}), we have
{\scriptsize
\begin{eqnarray*}
\left(
\begin{array}{cccc}
	0	& 0	& 0	& 0\\
	0	& x+\frac{1}{2}	& x+1	& x+\frac{3}{2}\\
	0	& x^2-\frac{1}{2}	& x^2+x-\frac{1}{2}	& x^2+2x\\
	0	& x^3-\frac{3}{2}x^2-x+\frac{3}{4}	& x^3-\frac{5}{2}x	& x^3+\frac{3}{2}x^2-\frac{5}{2}x-\frac{3}{2}
\end{array}
\right)
=
\left(
\begin{array}{cccc}
1	&  0  & 0	        & 0\\
0	& -1  & 0	        & 0\\
0	&  0  &  \frac{1}{2}& 0\\
0	& 0	  & 0	        & -\frac{1}{6}
\end{array}
\right)^{-1}
\left(
\begin{array}{cccc}
0	& 0	& 0	& 0\\
0	& 1	& 0	& 0\\
0	& 1	& \frac{1}{2}	& 0\\
0	& 1	& 1	& \frac{1}{6}
\end{array}
\right)
\times
\\
\left(
\begin{array}{cccc}
0	& 1	& 1	& 1\\
0	& -x-\frac{1}{2}	& -x-1	& -x-\frac{3}{2}\\
0	& x^2+2x+\frac{1}{2}	& x^2+3x+\frac{3}{2}	& x^2+4x+3\\
0	& -x^3-\frac{9}{2}x^2-5x-\frac{3}{4}	& -x^3-6x^2-\frac{19}{2}x-3	& -x^3-1\frac{5}{2}x^2-\frac{31}{2}x-\frac{15}{2}
\end{array}
\right).
\end{eqnarray*}
}

\noindent
Kim et al.  \cite[Theorean 2.11]{Kimetal2014C} introduced the following result. For  $n\in \mathbb{Z},\; k\in \mathbb{N}$,
\begin{equation} \label{eq:3.15}
(-1)^n \frac{Ch_n^{(k)}(x)}{n!}= \sum_{m=1}^n \frac{\binom{n1-}{n-m}}{m!}  \widehat{Ch}_m^{(k)}(-x).
\end{equation}

\noindent
We can write Eq. (\ref{eq:3.15}) in the matrix form as follows
\begin{equation} \label{eq:3.16}
{\bf Ch}^{(k)} (x)=  {(\bf D_n)}^{-1} {\bf C}\; {\widehat{\bf Ch}}^{(k)}(-x).
\end{equation}

\noindent
For example, if setting  $0\leq n \leq 3,\; 0 \leq k \leq n$ in (\ref{eq:3.16}), we have
{\scriptsize
\begin{eqnarray*}
	\left(
	\begin{array}{cccc}
		0	& 0	& 0	& 0 \\
		0	& x-\frac{1}{2} &	x-1	& x-\frac{3}{2} \\
		0	& x^2-2x+\frac{1}{2} &	x^2-3x+\frac{3}{2} &	x^2-4x+3\\
		0	& x^3-\frac{9}{2}x^2+5x-\frac{3}{4} &	x^3-6x^2+\frac{19}{2}x-3 &	x^3-\frac{15}{2}x^2+\frac{3}{2}1x-\frac{15}{2}
	\end{array}
	\right)=
\left(
\begin{array}{cccc}
1	&  0  & 0	        & 0\\
0	& -1  & 0	        & 0\\
0	&  0  &  \frac{1}{2}& 0\\
0	& 0	  & 0	        & -\frac{1}{6}
\end{array}
\right)^{-1} \times
\\
\left(
\begin{array}{cccc}
0	& 0	& 0	& 0\\
0	& 1	& 0	& 0\\
0	& 1	& \frac{1}{2}	& 0\\
0	& 1	& 1	& \frac{1}{6}
\end{array}
\right) 
\left(
\begin{array}{cccc}
0	& 1	&  1	& 1\\
0	&-x+\frac{1}{2} &	-x+1 &	-x+\frac{3}{2} \\
0	& x^2-\frac{1}{2} &	x^2-x-\frac{1}{2} &	x^2-2x \\
0	& -x^3-\frac{3}{2}x^2+x+\frac{3}{4} &	-x^3+\frac{5}{2}x	& -x^3+\frac{3}{2}x^2+\frac{5}{2}x-\frac{3}{2}
\end{array}
\right).
\end{eqnarray*}
}

\noindent
Kim et al.  \cite[Eq. (2.32)]{Kimetal2014C} introduced the following result
\begin{equation} \label{eq:3.17}
\sum_{n=0}^\infty \widehat{Ch}_n^{(k)} (x) {t^n \over n!} =(1+t)^{x+k} \left( {2 \over 2+t}\right)^k.
\end{equation}
We can derive the relation between the first and the second kind Changhee polynomials with higher order as follows. 

\begin{theorem}
For $n,k \in \mathbb{N}$, then
\begin{equation}\label{eq:3.18}
\widehat{Ch}_n^{(k)} (x)=
\sum_{n=0}^j \binom{j}{n}   {k! \over (k+n-j)!} Ch_n^{(k)} (x).
\end{equation}
\end{theorem}

\begin{proof}
From Eq. (\ref{eq:3.17}), 
\begin{eqnarray} \label{eq:3.19}
\sum_{n=0}^\infty \widehat{Ch}_n^{(k)} (x) {t^n \over n!} & = &
(1+t)^x \left( {2 \over 2+t}\right)^k (1+t)^k
\nonumber\\
& = &
\sum_{n=0}^\infty Ch_n^{(k)} (x) {t^n \over n!} (1+t)^k
\nonumber\\
& = &
\sum_{n=0}^\infty Ch_n^{(k)} (x) {t^n \over n!} \sum_{i=0}^\infty \binom{k}{i} t^i
\nonumber\\
& = &
\sum_{n=0}^\infty \sum_{i=0}^\infty \binom{k}{i} Ch_n^{(k)} (x)  {t^{n+i} \over n!}   
\nonumber\\
& = &
\sum_{n=0}^\infty \sum_{i=0}^\infty \binom{k}{i} Ch_n^{(k)} (x)  {(n+i)! \over n!}{t^{n+i} \over (n+i)!}   
\nonumber\\
& = &
\sum_{n=0}^\infty \sum_{j=n}^\infty \binom{k}{j-n} Ch_n^{(k)} (x)  {j! \over n!}{t^j \over j!}   
\nonumber\\
& = &
\sum_{j=0}^\infty \sum_{n=0}^j \binom{k}{j-n} Ch_n^{(k)} (x)  {j! \over n!}{t^j \over j!},
\end{eqnarray}

\noindent
by equating the coefficients of $t^j$ on both sides, we get Eq. (\ref{eq:3.18}), this completes the proof.
\end{proof}

\noindent
Setting $x=0$, in Eq. (\ref{eq:3.19}), we get the relation between the first and the second kinds of the Changhee numbers with higher orders, 
\begin{equation}\label{eq:3.20}
\widehat{Ch}_m^{(k)} =
\sum_{n=0}^m \binom{m}{n}   {k! \over (k+n-m)!}  Ch_n^{(k)}.
\end{equation}

\noindent
We obtain the relation between the first and the second kinds of the Changhee numbers by setting $k=1$, in Eq. (\ref{eq:3.20}), we get
\begin{equation}\label{eq:3.21}
\widehat{Ch}_m =
\sum_{n=0}^m \binom{m}{n}   {1 \over (n+1-m)!}  Ch_n.
\end{equation}

\noindent
Kim et al. \cite[Eq. 2.16]{Kimetal2014C} derived the following relation,
\begin{equation} \label{eq:3.22}
\sum_{n=0}^\infty Ch_n^{(k)} (x) { t^n \over n!} =\left({2 \over {2+t} } \right)^k (1+t)^x.
\end{equation}

\begin{theorem}
Forn $m,k \in \mathbb{N}$, we have
\begin{equation} \label{eq:3.23}
Ch_m^{(k)} (x)=\sum_{n=0}^m (-1)^{m-n} \binom{m}{n} {(k+m-n-1)! \over (k-1)!} \widehat{Ch}_n^{(k)} (x).
\end{equation}
\end{theorem}

\begin{proof}
From Eq. (\ref{eq:3.22}) we have
\begin{eqnarray} \label{eq:3.24}
\sum_{n=0}^\infty Ch_n^{(k)} (x) { t^n \over n!} & = &
\left({2 \over {2+t} } \right)^k (1+t)^{x+k} (1+t)^{-k}
\nonumber\\
&= &
\sum_{n=0}^\infty \widehat{Ch}_n^{(k)} (x) {t^n \over n!} \sum_{i=0}^\infty (-1)^i \binom{k+i-1}{i} t^i
\nonumber\\
&= &
\sum_{n=0}^\infty \sum_{i=0}^\infty (-1)^i \binom{k+i-1}{i} \widehat{Ch}_n^{(k)} (x) {t^{n+i} \over n!} 
\nonumber\\
&= &
\sum_{n=0}^\infty \sum_{m=n}^\infty (-1)^{m-n} \binom{k+m-n-1}{m-n} \widehat{Ch}_n^{(k)} (x) {t^m \over n!} 
\nonumber\\
&= &
\sum_{n=0}^\infty \sum_{m=n}^\infty (-1)^{m-n} {m! \over n!} \binom{k+m-n-1}{m-n} \widehat{Ch}_n^{(k)} (x) {t^m \over m!} 
\nonumber\\
&= &
\sum_{m=0}^\infty \sum_{n=0}^m (-1)^{m-n} {m! \over n!} \binom{k+m-n-1}{m-n} \widehat{Ch}_n^{(k)} (x) {t^m \over m!},
\end{eqnarray}
by equating the coefficients of $t^m$, on both sides, we get Eq. (\ref{eq:3.23}), this completes the proof.
\end{proof}

\noindent
Hence if $x=0$, then
\begin{equation} \label{eq:3.25}
Ch_m^{(k)} =\sum_{n=0}^m (-1)^{m-n} \binom{m}{n} {(k+m-n-1)! \over (k-1)!} \widehat{Ch}_n^{(k)} .
\end{equation}

\noindent
Moreover, if $k=1$,
\begin{equation} \label{eq:3.26}
Ch_m =\sum_{n=0}^m (-1)^{m-n} {m! \over n!} \widehat{Ch}_n.
\end{equation}

\noindent
We can find the relation between Changhee polynomials of the second order with Daehee and Euler's polynomials.

\noindent
From Eq. (\ref{eq:2.2}) and Eq. (\ref{eq:3.9}), we have
\begin{equation} \label{eq:3.27}
\widehat{Ch}_n^{(k)}(x)=\sum_{\ell=0}^n (-1)^\ell  \binom{n}{\ell} D_{n-\ell}^{(\ell)} E_\ell^{(k)} (-x).
\end{equation}

\noindent
But
\[
D_n^{(k)}=\sum_{\ell=0}^n s_1(n,\ell) B_\ell^{(k)},
\]

\noindent
therefore
\begin{equation} \label{eq:3.28}
\widehat{Ch}_n^{(k)}(x)=\sum_{\ell=0}^n \sum_{m=0}^{n-\ell} (-1)^\ell \binom{n}{\ell} s_1(n-\ell,m) B_m^{(\ell)} E_\ell^{(k)} (-x).
\end{equation}



\end{document}